\documentclass{article}
\usepackage{amssymb}
\usepackage{amsfonts}

\newtheorem{theorem}{Theorem}

\newtheorem{proposition}[theorem]{Proposition}

\newenvironment{proof}[1][Proof]{\noindent\textbf{#1.} }{\ \rule{0.5em}{0.5em}}
\input{tcilatex}

\begin{document}

\title{A Note on Global Suprema of Band-Limited Spherical Random Functions }
\author{Domenico Marinucci\thanks{%
Corresponding author; email address marinucc@mat.uniroma2.it. Research
supported by the ERC Grants n. 277742 \emph{Pascal,} "Probabilistic and
Statistical Techniques for Cosmological Applications".} and Sreekar
Vadlamani \\
Department of Mathematics, University of Rome Tor Vergata and\\
Tata Institute for Fundamental Research, Bangalore}
\maketitle

\begin{abstract}
In this note, we investigate the behaviour of suprema for band-limited
spherical random fields. We prove upper and lower bound for the expected
values of these suprema, by means of metric entropy arguments and discrete
approximations; we then exploit the Borell-TIS inequality to establish
almost sure upper and lower bounds for their fluctuations. Band limited
functions can be viewed as restrictions on the sphere of random polynomials
with increasing degrees, and our results show that fluctuations scale as the
square root of the logarithm of these degrees.

\begin{itemize}
\item Keywords and Phrases: Spherical Random Fields, Suprema, Metric
Entropy, Almost Sure Convergence

\item AMS Classification: 60G60; 62M15, 53C65, 42C15
\end{itemize}
\end{abstract}

\section{Introduction}

The analysis of the behaviour of suprema of Gaussian processes is one of the
classical topics in probability theory (\cite{RFG},\cite{azaisbook}); in
this note, we shall be concerned with suprema of band-limited random fields
defined on the unit sphere $S^{2}$. More precisely, let $T:S^{2}\times
\Omega \rightarrow \mathbb{R}$ be a measurable zero mean, finite variance
Gaussian field defined on for some probability space $\left\{ \Omega ,\Im
,P\right\} ;$ we assume $T(.)$ is isotropic, e.g. the vectors%
\[
\left\{ T(x_{1}),...,T(x_{k})\right\} \text{ and }\left\{
T(gx_{1}),...,T(gx_{k})\right\} 
\]%
have the same law, for all $k\in \mathbb{N}$, $x_{1},...,x_{k}\in S^{2}$ and 
$g\in SO(3),$ the group of rotations in $\mathbb{R}^{3}.$ It is then known
that the field $\left\{ T(.)\right\} $ is necessarily mean square continuous
(\cite{mp2012}) and the following spectral representation holds:%
\[
T(x)=\sum_{\ell =0}\sum_{m=-\ell }^{\ell }a_{\ell m}Y_{\ell m}(x), 
\]%
where the \emph{spherical harmonics }$\left\{ Y_{\ell m}\right\} $ form an
orthonormal systems of eigenfunctions of the spherical Laplacian, $\Delta
_{S^{2}}Y_{\ell m}=-\ell (\ell +1)Y_{\ell m}$ (see \cite{steinweiss},\cite%
{marpecbook})$,$ while the random coefficients $\left\{ a_{\ell m}\right\} $
form a triangular array of complex-valued, zero-mean, uncorrelated Gaussian
variables with variance $E\left\vert a_{\ell m}\right\vert ^{2}=C_{\ell },$
the \emph{angular power spectrum of the field. }In the sequel, we shall
adopt the following general model for the behavious of $\left\{ C_{\ell
}\right\} ;$ as $\ell \rightarrow \infty ,$ there exist $\alpha >2$ and a
positive rational function $G(\ell )$ such that 
\begin{equation}
C_{\ell }=G(\ell )\ell ^{-\alpha },\,\,\,\,0<c_{1}<G(\ell )<c_{2}<\infty
\label{usucon}
\end{equation}%
Spherical random fields have recently drawn a lot of applied interest,
especially in an astrophysical environment (see \cite{bennett2012}, \cite%
{marpecbook}); closed form expressions for the density of their maxima and
for excursion probabilities have been given in (\cite{ChengSchwar},\cite%
{chengxiao},\cite{MarVad}). In particular, the latter references exploit the
Gaussian Kinematic Fundamental formula by Adler and Taylor (see \cite{RFG})
to approximate excursion probabilities by means of the expected value of the
Euler-Poincar\`{e} characteristic for excursion sets. It is then easy to
show that%
\[
{\mathbb{E}}\mathcal{L}_{0}(A_{u}(T))=2\left\{ 1-\Phi (u)\right\} +4\pi
\left\{ \sum_{\ell }\frac{2\ell +1}{4\pi }C_{\ell }\frac{\ell (\ell +1)}{2}%
\right\} \frac{u\phi (u)}{\sqrt{(2\pi )^{3}}}, 
\]%
where $\phi ,\Phi $ denote density and distribution function of a standard
Gaussian variable, while $A_{u}(T):=\left\{ x\in S^{2}:T(x)\geq u\right\} .$
It is also an easy consequence of results in Ch.14 of \cite{RFG} that there
exist $\alpha >1$ and $\mu ^{+}>0$ such that, for all $u>\mu ^{+}$%
\begin{equation}
\left\vert {\mathbb{P}}\left\{ \sup_{x\in S^{2}}T(x)>u\right\} -2\left\{
(1-\Phi (u)+u\phi (u)\lambda \right\} \right\vert \leq \left\{ 4\pi \lambda
\right\} \exp (-\frac{\alpha u^{2}}{2}),  \label{before}
\end{equation}
where%
\[
\lambda :=\sum_{\ell }\frac{2\ell +1}{4\pi }C_{\ell }\frac{\ell (\ell +1)}{2}%
, 
\]%
denotes the derivative of the covariance function at the origin, see again (%
\cite{ChengSchwar},\cite{chengxiao},\cite{MarVad}).

When working on compact domains as the sphere, it is often of great interest
to focus on sequences of band-limited random fields; for instance, a very
powerful tool for data analysis is provided by fields which can be viewed as
a sequence of wavelet transforms (at increasing frequencies) of a given
isotropic spherical field $T.$ More precisely, take $b(.)$ to be a $%
C^{\infty }$ function, compactly supported in $[\frac{1}{2},2];$ having in
mind the wavelets interpretation, it would be natural to impose the
partition of unity property $\sum_{\ell }b^{2}(\frac{\ell }{2^{j}})\equiv 1,$
but this condition however plays no role in our results to follow. Let us
now focus on the sequence of band-limited spherical random fields%
\[
\beta _{j}(x):=\sum_{\ell =2^{j-1}}^{2^{j+1}}b(\frac{\ell }{2^{j}})a_{\ell
m}Y_{\ell m}(x), 
\]%
which have a clear interpretation as wavelet components of the original
field, and as such lend themselves to a number of statistical applications,
see for instance \cite{bkmpAoS}, \cite{cammar},\cite{npw1},\cite{pietrobon1}%
. Band-limited spherical fields have also been widely studied in other
context of mathematical physics, although in such cases $b(.)$ is not
necessarily assumed to be smooth, see for instance \cite{zelditch} and the
references therein.

In the sequel, it will be convenient to normalize the variance of $\left\{
\beta _{j}(x)\right\} $ to unity, and thus focus on 
\[
\widetilde{\beta }_{j}(x):=\frac{\beta _{j}(x)}{\sqrt{\sum_{\ell }b^{2}(%
\frac{\ell }{2^{j}})\frac{2\ell +1}{4\pi }C_{\ell }}}\text{ .}
\]%
The sequence of fields $\left\{ \widetilde{\beta }_{j}(x)\right\} $ has
covariance functions%
\[
\rho _{j}(x,y)=\frac{\sum_{\ell }b^{2}(\frac{\ell }{2^{j}})\frac{2\ell +1}{%
4\pi }C_{\ell }P_{\ell }(\left\langle x,y\right\rangle )}{\sum_{\ell }b^{2}(%
\frac{\ell }{2^{j}})\frac{2\ell +1}{4\pi }C_{\ell }}
\]%
and second spectral moments%
\[
\lambda _{j}:=\frac{\sum_{\ell }b^{2}(\frac{\ell }{2^{j}})\frac{2\ell +1}{%
4\pi }C_{\ell }P_{\ell }^{\prime }(1)}{\sum_{\ell }b^{2}(\frac{\ell }{2^{j}})%
\frac{2\ell +1}{4\pi }C_{\ell }}=\frac{\sum_{\ell }b^{2}(\frac{\ell }{2^{j}})%
\frac{2\ell +1}{4\pi }C_{\ell }\frac{\ell (\ell +1)}{2}}{\sum_{\ell }b^{2}(%
\frac{\ell }{2^{j}})\frac{2\ell +1}{4\pi }C_{\ell }},
\]%
see \cite{MarVad}. For fixed $j,$ as in (\ref{before}) it follows from
results in \cite{RFG} that there exist $\alpha >1$ and $\mu ^{+}>0$ such
that, for all $u>\mu ^{+}$%
\begin{equation}
\left\vert {\mathbb{P}}\left\{ \sup_{x\in S^{2}}\widetilde{\beta }%
_{j}(x)>u\right\} -2\left\{ (1-\Phi (u)+u\phi (u)\lambda _{j}\right\}
\right\vert \leq \left\{ 4\pi \lambda _{j}\right\} \exp (-\frac{\alpha u^{2}%
}{2}).  \label{sabato}
\end{equation}%
However, here for $j\rightarrow \infty $ we also have $\lambda
_{j}\rightarrow \infty ,$ whence the previous result clearly becomes
meaningless. Intuitively, sample paths become rougher and rougher as $j$
grows, hence any fixed threshold is crossed with probability tending to one.
In \cite{MarVad}, uniform bounds for band-limited fields have indeed been
established, covering even nonGaussian circumstances; however these bounds
require a further averaging in the space domain for the fields considered,
and this averaging ensures the uniform boundedness of $\lambda _{j};$ in
these circumstances, the multiplicative constant on the right-hand side of (%
\ref{sabato}) can be simply incorporated into the exponential choosing a
different constant $1<\alpha ^{\prime }<\alpha .$

There is, however, a question that naturally arises for the cases where $%
\lambda _{j}$ diverges - e.g., whether it is possible to provide bounds on
global suprema, allowing the thresholds to grow with frequency. This is a
natural question for a number of statistical applications, for instance when
considering thresholding estimates or multiple testing. Loosely speaking,
the issue we shall be concerned with is then related to the existence of a
growing sequence $\tau _{j}$ and positive constants $c_{1},c_{2},c_{1}^{%
\prime },c_{2}^{\prime }$ such that 
\[
c_{1}\leq {\mathbb{E}}\left\{ \frac{\sup_{x\in S^{2}}\widetilde{\beta }%
_{j}(x)}{\tau _{j}}\right\} \leq c_{2},\text{ and /or }{\mathbb{P}}\left\{
c_{1}^{\prime }\leq \frac{\sup_{x\in S^{2}}\widetilde{\beta }_{j}(x)}{\tau
_{j}}\leq c_{2}^{\prime }\right\} =1 
\]%
In fact, we shall be able to be more precise with our lower bounds. To make
this statement more precise, it will be convenient to write $\ell
_{j}:=2^{j};$ we shall then establish the following

\begin{theorem}
There exist positive constants $\gamma _{1},\gamma _{2}\geq 1,$ such that%
\[
1\leq \liminf_{j}\frac{{\mathbb{E}}\left\{ \sup_{x\in S^{2}}\widetilde{\beta 
}_{j}(x)\right\} }{\sqrt{4\log \ell _{j}}}\leq \limsup_{j}\frac{{\mathbb{E}}%
\left\{ \sup_{x\in S^{2}}\widetilde{\beta }_{j}(x)\right\} }{\sqrt{4\log
\ell _{j}}}\leq \gamma _{1}, 
\]%
and%
\[
{\mathbb{P}}\left( 1\leq \liminf_{j}\frac{\left\{ \sup_{x\in S^{2}}%
\widetilde{\beta }_{j}(x)\right\} }{\sqrt{4\log \ell _{j}}}\leq \limsup_{j}%
\frac{\left\{ \sup_{x\in S^{2}}\widetilde{\beta }_{j}(x)\right\} }{\sqrt{%
4\log \ell _{j}}}\leq \gamma _{2}\right) =1. 
\]
\end{theorem}

The corresponding upper bounds are proved in Section 2, while the proofs for
the lower bounds are collected in Section 3.

The random functions $\left\{ \widetilde{\beta }_{j}(.)\right\} $ can be
viewed as restrictions to the sphere of linear combinations of polynomials
with increasing degree $p_{j}=2\ell _{j}$ (\cite{marpecbook})$.$ Our results
can then be summarized by simply stating that as $j\rightarrow \infty ,$ the
supremum of $\left\{ \widetilde{\beta }_{j}(.)\right\} $ grows as twice the
square root of the logarithm of $p_{j}$.

\section{Metric Entropy and Upper Bounds}

The result we shall give in this Section is the following.\bigskip

\begin{proposition}
There exist a positive constant $c$ such that, for all $j\in \mathbb{N}$%
\begin{equation}
{\mathbb{E}}\left\{ \sup_{x\in S^{2}}\widetilde{\beta }_{j}(x)\right\} \leq c%
\sqrt{4\log \ell _{j}}.  \label{1.1}
\end{equation}%
Moreover there exist another positive constant $C$ such that%
\begin{equation}
{\mathbb{P}}\left( \limsup_{j}\frac{\left\{ \sup_{x}\widetilde{\beta }%
_{j}(x)\right\} }{\sqrt{4\log \ell _{j}}}\leq C+\frac{1}{\sqrt{2}}\right) =1.
\label{1.3}
\end{equation}
\end{proposition}

\begin{proof}
(\ref{1.1}) Define the canonical (Dudley) metric on $S^{2}$ as follows: 
\[
d_{j}(x,y)=\sqrt{{\mathbb{E}}\left( \widetilde{\beta }_{j}(x)-\widetilde{%
\beta }_{j}(y)\right) ^{2}}=\sqrt{2-2\rho _{j}(x,y)},
\]%
see \cite{RFG}. Note that since $\widetilde{\beta }_{j}(x)$ is isotropic,
all the distances can be measured from one fixed point (say the north pole).
Therefore, 
\[
d_{j}^{2}(x,y)=2(1-\rho _{j}(\langle \cos \theta \rangle ))
\]%
where $\theta :=\arccos \left\langle x,y\right\rangle $ is the usual
geodesic distance on the sphere, and%
\[
1-\rho _{j}(\cos \theta )=\frac{\sum_{\ell =2^{j-1}}^{2^{j+1}}b^{2}\left( 
\frac{\ell }{2^{j}}\right) \frac{(2l+1)}{4\pi }C_{\ell }\left( 1-P_{\ell
}(\cos \theta )\right) }{\sum_{\ell =2^{j-1}}^{2^{j+1}}b^{2}\left( \frac{%
\ell }{2^{j}}\right) \frac{(2\ell +1)}{4\pi }C_{\ell }}.
\]%
Now fix $\theta <1/(K\ell )$ and use Hilb's asymptotics (\cite{szego}) to
obtain%
\[
P_{\ell }(\cos \theta )=\sqrt{\frac{\theta }{\sin \theta }}\,\,J_{0}((\ell
+1/2)\theta )+\delta (\theta ),\,\,\,\,\,\,\,\,\,\,\delta (\theta )=O(\theta
^{2}),
\]%
where 
\[
J_{0}(z):=\sum_{k=1}^{\infty }(-1)^{k}\frac{x^{2k}}{2^{2k}(k!)^{2}}\text{ ,}
\]%
is the standard Bessel function of zeroth order; note also that 
\[
\lim_{K\rightarrow \infty }\sup_{\theta \leq (K\ell )^{-1}}\left\vert \frac{%
1-J_{0}((\ell +1/2)\theta )}{\ell ^{2}\theta ^{2}}-\frac{1}{4}\right\vert
=\lim_{K\rightarrow \infty }\sup_{x\leq K^{-1}}\left\vert \frac{1-J_{0}(x)}{%
x^{2}}-\frac{1}{4}\right\vert =0,
\]%
which means that for all $\delta >0,$ there exist $K_{\delta }$ small enough
so that%
\[
(\frac{1}{4}-\delta )\ell ^{2}\theta ^{2}\leq 1-J_{0}((\ell +1/2)\theta
)\leq (\frac{1}{4}-\delta )\ell ^{2}\theta ^{2},\,\,\text{ for all }\theta <%
\frac{K_{\delta }}{\ell }.
\]%
Combining these bounds with Hilb's asymptotics, we get for $\theta <\frac{%
K_{\delta }}{\ell }$ 
\[
(\frac{1}{4}-\delta )\ell ^{2}\theta ^{2}+O(\theta ^{2})\leq 1-P_{\ell
}(\cos \theta )\leq (\frac{1}{4}+\delta )\ell ^{2}\theta ^{2}+O(\theta ^{2}).
\]%
It follows that%
\begin{eqnarray*}
1-\rho _{j}(\cos \theta ) &=&\frac{\sum_{\ell =2^{j-1}}^{2^{j+1}}b^{2}\left( 
\frac{\ell }{2^{j}}\right) \frac{(2l+1)}{4\pi }C_{\ell }\left( 1-P_{\ell
}(\cos \theta )\right) }{\sum_{\ell =2^{j-1}}^{2^{j+1}}b^{2}\left( \frac{%
\ell }{2^{j}}\right) \frac{(2\ell +1)}{4\pi }C_{\ell }} \\
&\leq &\frac{1}{4}\theta ^{2}(2^{j+1}+1/2)^{2}+O(\theta ^{2})
\end{eqnarray*}%
and likewise%
\[
1-\rho _{j}(\cos \theta )\geq \frac{1}{4}\theta
^{2}(2^{j-1}+1/2)^{2}+O(\theta ^{2}),
\]%
thus implying that, for some constants $c_{1},c_{2}>0$ 
\[
c_{1}\theta ^{2}2^{2j}\leq 1-\rho _{j}(\cos \theta )\leq c_{2}\theta
^{2}2^{2j}.
\]%
We hence get%
\[
c_{1}^{\prime }\theta ^{2}\leq \frac{d_{j}^{2}(0,(\theta ,\phi ))}{\ell
_{j}^{2}}\leq c_{2}^{\prime }\theta ^{2},
\]%
and more generally for $\xi _{1},\xi _{2}\in S^{2}$%
\[
c_{1}^{\prime }d_{S}^{2}(\xi _{1},\xi _{2})\leq \frac{d_{j}^{2}(\xi _{1},\xi
_{2})}{\ell _{j}^{2}}\leq c_{2}^{\prime }d_{S}^{2}(\xi _{1},\xi _{2}),
\]%
where $d_{S}(\xi _{1},\xi _{2}):=\arccos (\left\langle \xi _{1},\xi
_{2}\right\rangle )$ is the standard spherical distance. Now for $%
\varepsilon <C$ and $\theta <\frac{C}{\ell }$, define the sequence of $d_{j}$%
-balls $B_{d_{j}}(\xi _{jk},\varepsilon )=\{u\in S^{2}:d_{j}(\xi
_{jk},u)\leq \varepsilon \},$ which can be rewritten as 
\[
B_{d_{j}}(\xi _{jk},\varepsilon )=\{\xi \in S^{2}:d_{j}(\xi _{jk},\xi )=\ell
_{j}d_{S}(\xi _{jk},\xi )\leq \varepsilon \}.
\]%
Hence $B_{d_{j}}(\xi _{jk},\varepsilon )$ is a spherical cap of radius $\sim 
\frac{\varepsilon }{\ell },$ with Euclidean volume 
\[
B_{d}(\xi _{jk},\varepsilon )\sim \frac{\varepsilon ^{2}}{\ell _{j}^{2}}.
\]%
It follows that the number of $d_{j}$-balls needed to cover the sphere is
asymptotic to $N_{j}(\varepsilon )\sim \frac{\ell _{j}^{2}}{\varepsilon ^{2}}%
.$ Consequently, by Theorem 1.3.3. of \cite{RFG}, for any $\delta \in (0,\pi
]$ there exists a universal constant $K^{\ast }$ such that 
\begin{eqnarray*}
{\mathbb{E}}\left( \sup_{x}\widetilde{\beta }_{j}(x)\right)  &\leq &K^{\ast
}\int_{0}^{\delta }\sqrt{\log N_{j}(\varepsilon )}\,d\varepsilon  \\
&\leq &K^{\ast }\left\{ \int_{0}^{C/\ell _{j}}\sqrt{\log N_{j}(\varepsilon )}%
\,d\varepsilon +\int_{C/\ell _{j}}^{\delta }\sqrt{\log N_{j}(\varepsilon )}%
\,d\varepsilon \right\} .
\end{eqnarray*}%
Clearly for $\varepsilon >C/\ell _{j}$ one has $N_{j}(\varepsilon )\leq
c\ell _{j}^{4},$ whence%
\[
\int_{C/\ell _{j}}^{\delta }\sqrt{\log N_{j}(\varepsilon )}\,d\varepsilon
\leq c^{\prime }\sqrt{4\log \ell _{j}}.
\]%
On the other hand 
\[
\int_{0}^{\delta }\sqrt{\log N_{j}(\varepsilon )}\,d\varepsilon
=\int_{0}^{\delta }\sqrt{2\log \left( \frac{\ell _{j}}{\varepsilon }\right) }%
\,d\varepsilon =\ell _{j}\int_{\sqrt{2\log \frac{\ell _{j}}{\delta }}%
}^{\infty }v^{2}\exp \left( -\frac{v^{2}}{2}\right) \,dv,
\]%
with the change of variables $\frac{\varepsilon }{\ell _{j}}=\exp \left(
-v^{2}/2\right) ,$ whence 
\begin{eqnarray*}
&&\int_{0}^{\delta }\sqrt{\log N_{j}(\varepsilon )}\,d\varepsilon  \\
&=&\ell _{j}\left( \left. (-1)v\exp \left( -\frac{v^{2}}{2}\right)
\right\vert _{\sqrt{2\log \frac{\ell _{j}}{\delta }}}^{\infty }+\int_{\sqrt{%
2\log \frac{l}{\delta }}}^{\infty }\exp \left( -\frac{v^{2}}{2}\right)
\,dv\right)  \\
&\leq &\ell _{j}\left( \frac{\delta }{\ell _{j}}\sqrt{2\log \left( \frac{%
\ell _{j}}{\delta }\right) }+\left( 2\log \left( \frac{\ell _{j}}{\delta }%
\right) \right) ^{-1/2}\frac{\delta }{\ell _{j}}\right)  \\
&=&\delta \left( \sqrt{2\log \left( \frac{\ell _{j}}{\delta }\right) }%
+\left( 2\log \left( \frac{\ell _{j}}{\delta }\right) \right) ^{-1/2}\right)
\leq c\sqrt{4\log \ell _{j}}.
\end{eqnarray*}%
Taking the same $C$ as in the entropy upper bound, and using the Borell-TIS
inequality (cf. \cite{RFG}) we have 
\begin{eqnarray*}
{\mathbb{P}}\left( \left\{ \sup_{x}\widetilde{\beta }_{j}(x)\right\}
>(C+\varepsilon )\sqrt{4\log \ell _{j}}\right)  &\leq &{\mathbb{P}}\left(
\Vert \widetilde{\beta }_{j}\Vert >E\Vert \widetilde{\beta }_{j}\Vert
+\varepsilon \sqrt{4\log \ell _{j}}\right)  \\
&\leq &\exp \left( -\frac{4\varepsilon ^{2}\log \ell _{j}}{2}\right)  \\
&=&\exp \left( -\log \ell _{j}^{2\varepsilon ^{2}}\right)  \\
&=&\frac{1}{\ell _{j}^{2\varepsilon ^{2}}}\rightarrow 0,\,\,\,\,\forall
\varepsilon >0
\end{eqnarray*}

Now for (\ref{1.3}), taking $\varepsilon >\frac{1}{\sqrt{2}}$ in the above
expression, we obtain summable probabilities, and then by a simple
application of the Borel-Cantelli Lemma we have that 
\[
{\mathbb{P}} \left( \limsup_{j}\frac{\Vert \widetilde{\beta }_{j}\Vert }{%
\sqrt{4\log \ell _{j}}}\geq C+\frac{1}{\sqrt{2}}\right) \leq
\lim_{j\rightarrow \infty }\sum_{j^{\prime }=j}^{\infty } {\mathbb{P}}
\left( \frac{\Vert \widetilde{\beta }_{j^{\prime }}\Vert }{\sqrt{4\log \ell
_{j^{\prime }}}}\geq C+\frac{1}{\sqrt{2}}\right) =0. 
\]
\end{proof}


\section{Discretization and Lower Bounds}

As explained in the Introduction, this Section is devoted to the proofs for
the lower bounds that follow.

\begin{proposition}
We have%
\begin{equation}
\liminf_{j}\frac{{\mathbb{E}}\left\{ \sup_{x\in S^{2}}\widetilde{\beta }%
_{j}(x)\right\} }{\sqrt{4\log \ell _{j}}}\geq 1.  \label{3h30b}
\end{equation}%
and%
\begin{equation}
{\mathbb{P}}\left( \liminf_{j}\frac{\left\{ \sup_{x\in S^{2}}\widetilde{%
\beta }_{j}(x)\right\} }{\sqrt{4\log \ell _{j}}}\geq 1\right) =1.
\label{3h30c}
\end{equation}
\end{proposition}

\begin{proof}
We start showing that, for all $\delta >0$, 
\begin{equation}
\lim_{j}{\mathbb{P}} \left( \frac{\left\{ \sup_{x\in S^{2}}\widetilde{\beta }%
_{j}\right\} }{\sqrt{4\log \ell _{j}}}>1-\delta \right) =1.  \label{3h30a}
\end{equation}%
Note first that $\sup \widetilde{\beta }_{j}\geq \sup_{k}\widetilde{\beta }%
_{j,k}$, where $\{\widetilde{\beta }_{j,k}\}$ is any discrete sample taken
from $\widetilde{\beta }_{j}$. Now, let us choose a grid of points such that
the distance between them is at least $2^{-j(1-\delta )},$ for some $\delta
>0$ - e.g., a $2^{-j(1-\delta )}$-net, see \cite{bkmpBer}. Note that the
vectors $\beta _{j,\cdot }$ and $\widetilde{\beta }_{j,\cdot }$ both have
cardinality of order $2^{2j(1-\delta )}$. By using the correlation
inequality given in Lemma 10.8 of \cite{marpecbook}, we have 
\[
{\mathbb{E}}\widetilde{\beta }_{j,k}\widetilde{\beta }_{j,k^{\prime }}\leq 
\frac{C}{(1+2^{j\delta })^{M}}, 
\]%
where $M\in \mathbb{N}$ can be chosen arbitrarily large. The idea of the
proof is to approximate these subsampled coefficients by means of a
triangular array of Gaussian i.i.d. random variables, say $\widehat{\beta }%
_{j,k}$. More precisely, let $\Sigma _{j}$ be the covariance matrix of the
Gaussian vector $\widetilde{\beta }_{j,\cdot }$; then define $\widehat{\beta 
}_{j,\cdot }=\Sigma _{j}^{-1/2}\widetilde{\beta }_{j,\cdot }$, which is
clearly a vector of i.i.d. Gaussian variables, and let $\lambda _{j,\mathrm{%
max}}$ and $\lambda _{j,\mathrm{min}}$ be the largest and the smallest
eigenvalues of the matrix $\Sigma _{j}$. Then 
\[
\lambda _{j,\mathrm{max}},\lambda _{j,\mathrm{min}}=1+O(\varepsilon _{j}), 
\]
for a deterministic sequence $\left\{ \varepsilon _{j}\right\} $ which goes
to zero faster than any polynomial (nearly exponentially). Indeed 
\begin{eqnarray*}
\lambda _{j,\mathrm{max}} &=&\sup_{x}x^{\prime }\Sigma
_{j}x=\sup_{x}x^{\prime }\left( \Sigma _{j}-I+I\right) x=\sup_{x}x^{\prime
}\left( \Sigma _{j}-I\right) x+1 \\
&\leq &\ell _{j}^{4(1-\delta )}\frac{C_{M}}{1+2^{j\delta M}}+1,
\end{eqnarray*}%
where the bound follows crudely from the cardinality of the off-diagonal
terms in the matrix. Similarly, 
\begin{eqnarray*}
\lambda _{j,\mathrm{min}} &=&\inf_{x}x^{\prime }\Sigma
_{j}x=\inf_{x}x^{\prime }\left( \Sigma _{j}-I+I\right) x \\
&=&\inf_{x}x^{\prime }\left( \Sigma _{j}-I\right) x+1\geq 1-\sup |x^{\prime
}\left( \Sigma _{j}-I\right) x| \\
&\geq &1-\ell _{j}^{4}\frac{C_{M}}{1+\ell _{j}^{\delta M}}.
\end{eqnarray*}%
As a consequence, writing $\Vert \cdot \Vert _{2}$ for the Euclidean inner
product in the appropriate dimension we have 
\begin{eqnarray*}
{\mathbb{E}}\left( \sup |\widetilde{\beta }_{j,k}-\widehat{\beta }%
_{j,k}|\right) &\leq &\sqrt{{\mathbb{E}}\Vert \widetilde{\beta }_{j,\cdot }-%
\widehat{\beta }_{j,\cdot }\Vert _{2}^{2}} \\
&=&\sqrt{{\mathbb{E}}\Vert (I-\Sigma ^{-1/2})\widetilde{\beta }_{j,\cdot
}\Vert _{2}^{2}}\leq |1-\lambda _{\mathrm{max}}|\sqrt{E\Vert \widetilde{%
\beta }_{j,\cdot }\Vert _{2}^{2}} \\
&\leq &\ell _{j}^{4(1-\delta )}\frac{C_{M}}{1+\ell _{j}^{\delta M}}\cdot
\ell _{j}^{2}=\ell _{j}^{2+4(1-\delta )}\frac{C_{M}}{1+\ell _{j}^{\delta M}}%
=O(\ell _{j}^{6-\delta M}).
\end{eqnarray*}%
We can now exploit a classical result by Berman (\cite{berman}) to conclude
that 
\[
{\mathbb{P}} \left( \left|\frac{\sup_{k}\widetilde{\beta }_{j,k}}{\sqrt{%
4\log \ell _{j}}}-1\right|>\varepsilon \right) \rightarrow 0, 
\]%
as $j\rightarrow \infty $, for all $\varepsilon >0$. Thus (\ref{3h30a}) is
established; (\ref{3h30b}) follows immediately, given that $\delta $ is
arbitrary. To establish (\ref{3h30c}), we use again the Borel-Cantelli
Lemma, so that we need to prove that, for all $\varepsilon >0$ 
\[
\sum_{j} {\mathbb{P}} \left( \frac{\left\{ \sup_{x\in S^{2}}\widetilde{\beta 
}_{j}\right\} }{\sqrt{4\log \ell _{j}}}<1-\varepsilon \right) <\infty . 
\]%
Clearly 
\[
{\mathbb{P}} \left( \frac{\left\{ \sup_{x\in S^{2}}\widetilde{\beta }%
_{j}\right\} }{\sqrt{4\log \ell _{j}}}<1-\varepsilon \right) \leq {\mathbb{P}%
} \left( \frac{\sup_{k}\widetilde{\beta }_{j,k}}{\sqrt{4\log \ell _{j}}}%
<1-\varepsilon \right) , \text{ for all }j, 
\]%
whence it suffices to prove that 
\[
\sum_{j}{\mathbb{P}} \left( \frac{\sup_{k}\widetilde{\beta }_{j,k}}{\sqrt{%
4\log \ell _{j}}}<1-\varepsilon \right) <\infty . 
\]%
Now%
\[
{\mathbb{P}} \left( \frac{\sup_{k}\widetilde{\beta }_{j,k}}{\sqrt{4\log \ell
_{j}}}<1-\varepsilon \right) ={\mathbb{P}} \left( \frac{\sup_{k}\left( 
\widehat{\beta }_{j,k}-\widehat{\beta }_{j,k}+\widetilde{\beta }%
_{j,k}\right) }{\sqrt{4\log \ell _{j}}}<1-\varepsilon \right) 
\]%
\begin{eqnarray*}
&\leq &{\mathbb{P}} \left( \frac{\sup_{k}\widehat{\beta }_{j,k}-\sup_{k}%
\left( \widehat{\beta }_{j,k}-\widetilde{\beta }_{j,k}\right) }{\sqrt{4\log
\ell _{j}}}<1-\varepsilon \right) \\
&\leq &{\mathbb{P}} \left( \frac{\sup_{k}\widehat{\beta }_{j,k}-\sup_{k}%
\left\vert \widehat{\beta }_{j,k}-\widetilde{\beta }_{j,k}\right\vert }{%
\sqrt{4\log \ell _{j}}}<1-\varepsilon \right) \\
&=&{\mathbb{P}} \left( \frac{\sup_{k}\widehat{\beta }_{j,k}}{\sqrt{4\log
\ell _{j}}}<1-\varepsilon +\frac{\sup_{k}\left\vert \widehat{\beta }_{j,k}-%
\widetilde{\beta }_{j,k}\right\vert }{\sqrt{4\log \ell _{j}}}\right) \\
&\leq &{\mathbb{P}} \left( \frac{\sup_{k}\widehat{\beta }_{j,k}}{\sqrt{4\log
\ell _{j}}}<1-\frac{\varepsilon }{2}\right) +{\mathbb{P}} \left( \frac{%
\sup_{k}\left\vert \widehat{\beta }_{j,k}-\widetilde{\beta }%
_{j,k}\right\vert }{\sqrt{4\log \ell _{j}}}>\frac{\varepsilon }{2}\right) \\
&\leq &{\mathbb{P}} \left( \frac{\sup_{k}\widehat{\beta }_{j,k}}{\sqrt{4\log
\ell _{j}}}<1-\frac{\varepsilon }{2}\right) +O(\ell _{j}^{6-\varepsilon M}).
\end{eqnarray*}%
The second term above is clearly summable, for all fixed $\varepsilon >0,$
by simply taking $M$ large enough. To check summability of the first term we
write%
\[
{\mathbb{P}} \left( \frac{\sup_{k}\widehat{\beta }_{j,k}}{\sqrt{4\log \ell
_{j}}}<1-\frac{\varepsilon }{2}\right) =\prod\limits_{k}{\mathbb{P}} \left( 
\frac{\widehat{\beta }_{j,k}}{\sqrt{4\log \ell _{j}}}<1-\frac{\varepsilon }{2%
}\right) 
\]
\begin{eqnarray*}
&=&\left( {\mathbb{P}} \left( \widehat{\beta }_{j,1}<(1-\frac{\varepsilon }{2%
})\sqrt{4\log \ell _{j}}\right) \right) ^{\ell _{j}^{2}} \\
&=&\left( 1-{\mathbb{P}} \left( \widehat{\beta }_{j,1}>(1-\frac{\varepsilon 
}{2})\sqrt{4\log \ell _{j}}\right) \right) ^{\ell _{j}^{2}} \\
&\leq &\left( 1-\frac{1}{(1-\frac{\varepsilon }{2})\sqrt{4\log \ell _{j}}}%
\left( 1-\frac{1}{(1-\frac{\varepsilon }{2})\sqrt{4\log \ell _{j}}}\right)
\cdot \frac{1}{\ell _{j}^{2(1-\frac{\varepsilon }{2})^{2}}}\right) ^{\ell
_{j}^{2}} \\
&\leq &\left( 1-\frac{1}{2(1-\frac{\varepsilon }{2})\sqrt{4\log \ell _{j}}}%
\cdot \frac{1}{\ell _{j}^{2(1-\frac{\varepsilon }{2})^{2}}}\right) ^{\ell
_{j}^{2}},
\end{eqnarray*}%
where we have used Mill's inequality for standard Gaussian variables, ${%
\mathbb{P}}\left\{ Z>z\right\} \geq \frac{z}{1+z^{2}}\phi (z).$ Since $(1-%
\frac{\varepsilon }{2})^{2}<1$, this term decays exponentially, and it is
hence summable. The proof of (\ref{3h30c}) is hence concluded.
\end{proof}


\begin{thebibliography}{99}
\bibitem{RFG} \textbf{Adler, R. J. and Taylor, J. E.}, \textbf{(2007)} \emph{%
Random Fields and Geometry}, Springer.

\bibitem{azaiswschebor} \textbf{Aza\"{\i}s, J.-M., Wschebor, M. (2005)} On
the Distribution of the Maximum of a Gaussian Field with d Parameters. \emph{%
Ann. Appl. Probab.} 15 no. 1A, 254--278.

\bibitem{azaisbook} \textbf{Aza\"{\i}s, J.-M., Wschebor, M. (2009)} Level
Sets and Extrema of Random Processes and Fields. John Wiley \& Sons, Inc.,
Hoboken, NJ.

\bibitem{bkmpAoS} \textbf{Baldi, P., Kerkyacharian, G., Marinucci, D. and
Picard, D. (2009) }Asymptotics for Spherical Needlets, \emph{Annals of
Statistics,} Vol. 37, No. 3, 1150-1171

\bibitem{bkmpBer} \textbf{Baldi, P., Kerkyacharian, G., Marinucci, D. and
Picard, D. (2009)} Subsampling Needlet Coefficients on the Sphere, \emph{%
Bernoulli},\emph{\ }Vol. 15, 438-463

\bibitem{bennett2012} \textbf{Bennett, C.L. et al. (2012) }Nine-Year WMAP
Observations: Final Maps and Results, arXiv:1212.5225

\bibitem{berman} \textbf{Berman, S. (1962)} A Law of Large Numbers for the
Maximum in a Stationary Gaussian Sequence, \emph{Ann.Math.Stat.}, 33, 1,
93-97

\bibitem{cammar} \textbf{Cammarota, V., Marinucci, D. (2014) }On the
Limiting Behaviour of Needlets Polyspectra, \emph{Ann.Inst.H.Poinc.,} in
press, arXiv:1307.4691

\bibitem{chengxiao} \textbf{Cheng, D. and Xiao, Y. (2012)} The Mean Euler
Characteristic and Excursion Probability of Gaussian Random Fields with
Stationary Increments, arXiv:1211.6693

\bibitem{ChengSchwar} \textbf{Cheng, D. and Schwartzman, A. (2013) }\
Distribution of the Height of Local Maxima of Gaussian Random Fields,
arXiv:1307.5863

\bibitem{dode2004} \textbf{Dodelson, S. (2003)} \emph{Modern Cosmology},
Academic Press

\bibitem{Durrer} \textbf{Durrer, R. (2008) }\emph{The Cosmic Microwave
Background, }Cambridge University Press.

\bibitem{mal} \textbf{Malyarenko, A. (2012)}, \emph{Invariant Random Fields
on Spaces with a Group Action}, Probability and its Applications, Springer.

\bibitem{marpecbook} \textbf{Marinucci, D. and Peccati, G. (2011) }\emph{%
Random Fields on the Sphere. Representation, Limit Theorem and Cosmological
Applications}, Cambridge University Press

\bibitem{mp2012} \textbf{Marinucci, D. and Peccati, G. (2012) }Mean Square
Continuity on Homogeneous Spaces of Compact Groups, arXiv:1210.7676.

\bibitem{MarVad} \textbf{Marinucci, D. and Vadlamani, S. (2013) }%
High-Frequency Asymptotics for Lipscitz-Killing Curvatures of Excursion Sets
on the Sphere, arXiv:1303.2456

\bibitem{npw1} \textbf{Narcowich, F.J., Petrushev, P. and Ward, J.D. (2006a)}
Localized Tight Frames on Spheres, \emph{SIAM Journal of Mathematical
Analysis }Vol. 38, pp. 574--594

\bibitem{pietrobon1} \textbf{Pietrobon, D., Amblard, A., Balbi, A., Cabella,
P., Cooray, A., Marinucci, D. (2008)} Needlet Detection of Features in WMAP
CMB Sky and the Impact on Anisotropies and Hemispherical Asymmetries, \emph{%
Physical Review D}, D78 103504

\bibitem{steinweiss} \textbf{Stein, E.M. and Weiss, G. (1971) }\emph{%
Introduction to Fourier Analysis on Euclidean Spaces. }Princeton University
Press

\bibitem{szego} \textbf{Szego, G. (1975)} Orthogonal Polynomials, 4th
edition, \textit{American Mathematical Society, Colloquium Publications,}
Vol. XXIII

\bibitem{tayloradler2009} \textbf{Taylor, J.E. and Adler, R.J. (2009)}
Gaussian Processes, Kinematic Formulae and Poincar\'{e}'s Limit. \emph{Ann.
Probab.} 37, no. 4, 1459--1482.

\bibitem{taylorvadlamani} \textbf{Taylor, J.E. and Vadlamani, S. (2013) }%
Random Fields and the Geometry of Wiener Space, \emph{Ann. Probab.,} 41, 4,
2724-2754, arXiv: 1105.3839

\bibitem{zelditch} \textbf{Zelditch, S. (2009) }Real and complex zeros of
Riemannian random waves, \emph{Contemp. Math.,} 484, 321--342
\end{thebibliography}
\end{document}